\documentclass[reqno,11pt]{amsart}
\usepackage[utf8]{inputenc}
\usepackage[margin=1.2 in]{geometry}
\usepackage{amsmath,amssymb,amsthm,mathrsfs,mathtools}
\usepackage{xcolor}
\usepackage{makecell}
\usepackage{cite}
\usepackage[colorlinks=true,citecolor=blue, linkcolor=red,urlcolor=blue]{hyperref}
\usepackage{enumitem}
\numberwithin{equation}{section}
\allowdisplaybreaks[2]

\theoremstyle{plain}
 \newtheorem{theorem}{Theorem}[section]
 \newtheorem{proposition}[theorem]{Proposition}
 \newtheorem{lemma}[theorem]{Lemma}
 
 \newtheorem{corollary}[theorem]{Corollary}
 \newtheorem{definition}[theorem]{Definition}
 
 \theoremstyle{remark}
 \newtheorem{remark}[theorem]{Remark}

\newcommand{\bal}{\begin{align}}
\newcommand{\bbal}{\begin{align*}}
\newcommand{\beq}{\begin{equation}}
\newcommand{\eeq}{\end{equation}}
\newcommand{\bca}{\begin{cases}}
\newcommand{\eca}{\end{cases}}
\def\div{\mathord{{\rm div}}}
\newcommand{\pa}{\partial}
\newcommand{\fr}{\frac}
\newcommand{\na}{\nabla}

\newcommand{\cd}{\cdot}
\newcommand{\ep}{\varepsilon}
\newcommand{\dd}{\mathrm{d}}
\newcommand{\B}{\dot{B}}

\newcommand{\R}{\mathbb{R}}

\newcommand{\les}{\lesssim}

\newcommand{\f}{\left}
\newcommand{\g}{\right}
\newcommand{\no}{\nonumber}

\numberwithin{equation}{section}

\makeatletter
	\@namedef{subjclassname@2020}{\textup{2020} Mathematics Subject
	Classification}
	\makeatother
\begin{document}

\title[ill-posedness of rotating Euler equations]{The Coriolis effect-induced $B^s_{\infty,r}$ and $B^s_{1,r}$ ill-posedness in the 3D rotating Euler equations}

\author{Jinlu Li}
\address{School of Mathematics and Computer Sciences, Gannan Normal University, Ganzhou 341000, China}
\email{lijinlu@gnnu.edu.cn}

\author{Yanghai Yu}
\address{School of Mathematics and Statistics, Anhui Normal University, Wuhu 241002, China}
\email{yuyanghai214@sina.com}

\begin{abstract}
In this paper, we establish the ill-posedness of the 3D rotating Euler equations in endpoint $L^{\infty}$ and $L^1$-based spaces. More precisely, our aim is two fold: (1) we prove an endpoint ill-posedness in any $B^s_{\infty,r}$ of the velocity formulation for the 3D rotating Euler equations. To the best of our knowledge, our work is the first one addressing the ill-posedness issue on the rotating Euler equations in any $L^\infty$-based framework without the vorticity formulation. These are in stark contrast to existing well-posedness results for the classical  Euler equations in endpoint Besov spaces $B^1_{\infty,1}(\R^3)$ and thus we identify a structural distinction between the 3D rotating and classical Euler equations. (2) we prove an endpoint ill-posedness in any $B^s_{1,r}$ of the vorticity formulation for the 3D rotating Euler equations, which seems to be the first ill-posedness result in $L^1$-type spaces although the currently-known results on the well-posedness were given in the Besov space $B_{1, 1}^{d+1}(\R^d)$ for the velocity equation and in the Sobolev space $W^{d,1}(\R^d)$ for the vorticity equation. The Coriolis effect-induced $B^s_{\infty,r}$ and $B^s_{1,r}$ ill-posedness phenomenon for the 3D rotating Euler equations, which may be the first work that precisely identifies a rigorous distinction in the theory of well-posedness on the 3D rotating and classical Euler equations.
\end{abstract}



\date{\today}
\maketitle
\section{Introduction}

In this paper, we consider the Cauchy problem for the 3D Euler equations with the Coriolis force, which is also called the rotating Euler equations and describes the motion of perfect incompressible fluids in the rotational framework
\begin{align}\label{CE}
\begin{cases}
\pa_t u+\Omega e_3\times u+u\cdot \nabla u+\nabla P=0, &\quad (t,x)\in \R^+\times\R^3,\\
\mathrm{div\,} u=0,&\quad (t,x)\in \R^+\times\R^3,\\
u(0,x)=u_0(x), &\quad x\in \R^3,
\end{cases}
\end{align}
where the vector field $u(t,x):[0,\infty)\times {\mathbb R}^d\to {\mathbb R}^d$ stands for the velocity of the fluid, the quantity $P(t,x):[0,\infty)\times {\mathbb R}^d\to {\mathbb R}$ denotes the scalar pressure, and $\mathrm{div\,} u=0$ means that the fluid is incompressible. The constant $\Omega\in\R$ represents the speed of rotation around
the vertical unit vector $e_3 = (0, 0, 1)$ and is called the Coriolis parameter. We call $\Omega e_3\times u=\Omega (-u_2,u_1,0)$ the Coriolis force, which plays a significant role in the large scale flows considered in meteorology and geophysics.
Problems concerning large-scale atmospheric and oceanic flows are known to be dominated by rotational effects. For this reason, almost all of the models of oceanography and meteorology dealing with large-scale phenomena include a Coriolis force. For example, an oceanic circulation featuring a hurricane is caused by the large rotation. The dispersive effect of rotation in fluid flows has been studied in the literature from various perspectives, see, e.g., geophysical flows \cite{IR,JM,JP}, life span and asymptotic behaviour in the case of fast rotations \cite{ch1,ch2,d} and almost global stability \cite{YGUO}.

From the PDE's point of view, it is crucial to know if an equation which models a physical phenomenon is well-posed in the Hadamard's sense: existence, uniqueness, and continuous dependence of the solutions with respect to the initial data. In particular, the lack of continuous dependence would cause incorrect solutions or non meaningful solutions. In this paper, we are interested in the lack of continuity of solution with respect to the initial data.

\subsection{Concept of mild ill-posedness}\label{subsec1.0} Let us first recall the rigorous definition of mild ill-posed proposed by Elgindi and Masmoudi \cite{E-ARMA}, which implies the discontinuity with respect to the initial data.
\begin{definition} Let $F(f)$ be a function of $f$. We say that a Cauchy problem
\begin{equation}\label{dp}
\begin{cases}
\pa_tf=F(f), \\
f(0,x)=f_0(x),
\end{cases}
\end{equation}
is mildly ill-posed in a Banach space $X$. If there exists a Banach space $Y$ continuously embedded in $X$ and a constant $c>0$ such that for any $\varepsilon, \delta>0$, there exists $f_0 \in Y$, with
$$
\left\|f_0\right\|_X \leq \varepsilon
$$
for which there exists a unique local solution $f(t) \in L^{\infty}([0, T] ; Y)$ for some $T>0$, and
$$
\|f(t)\|_X \geq c
$$
for some $0<t<\delta$.

If $c$ can be taken to be equal to $\fr1\varepsilon$, we will say that the Cauchy problem \eqref{dp}
is strongly ill-posed in the space $X$.
\end{definition}

\subsection{Previous Results and Motivations}
A fundamental challenge in mathematical physics is to understand the behavior of solutions for the complex rotating fluids.
There is by now an extensive literature on the wellposedness theory (especially for existence and uniqueness) for the Euler type equations.
\begin{itemize}
  \item {\bf The classical Euler equations}
\end{itemize}
In the case of $\Omega=0$, \eqref{CE} reduces to the original Euler equations. Kato \cite{Kato} and Kato-Ponce \cite{KatoP} established the local well-posedness of solution to Euler equations in the Sobolev type spaces. Vishik \cite{Varma,Vasns} constructed global solutions to 2D Euler in Besov space $B_{p, 1}^{2 / p+1}(\mathbb{R}^2)$ with $1<p<\infty$. For general dimension $d \geq 2$,
Chae \cite{Chae} obtained the local well-posedness of the Euler equations in critical Besov space $B_{p, 1}^{d / p+1}(\mathbb{R}^d)$ with $1<p<\infty$ or $s=d/ p+1,1<p<\infty, r=1$. We also refer to see the monographs \cite[Theorem 7.1]{B} for the local well-posedness of Euler equations in Besov spaces $B_{p, q}^s(\mathbb{R}^d)$. However, these two kinds of function spaces are only in the $L^p(1<p<\infty)$-framework since the Riesz transform is not bounded on $L^{\infty}$. For the case $p = \infty$ the unique local solvability in $B^1_{\infty,1}(\R^d)$ was proved by Pak and Park in \cite{Pak1} by using a compactness argument. Pak and Park \cite{PP} extended their analysis to cover the case $p = 1$, i.e., the Besov space $B^{d+1}_{1,1}(\R^d)$. Guo-Li-Yin \cite{GLY} proved the continuous dependence of the Euler equations in the space $B_{p, q}^s(\mathbb{R}^3)$ with $s>3 / p+1$, $1\leq p\leq \infty, 1 \leq q < \infty$ or $s=3 / p+1,1\leq p\leq \infty, q=1$. Furthermore, Bourgain and Li \cite{Bourgain1,Bourgain2,BLimrn} proved the strong local ill-posedness of the Euler equations in borderline Besov spaces $B^{d/p+1}_{p,r}$ with $(p,r)\in[1,\infty)\times(1,\infty]$ when $d=2,3$. For the end-point critical Sobolev space, Danchin and Jeong \cite{dj} proved the global propagation of the $W^{d,1}(\R^d)$ regularity for the vorticity. Recentaly, Shikh Khalil \cite{SK} settled  the ill-posedness in the critical Besov $B^1_{\infty,r}$ with $1 < r < \infty$. Using completely different methods, Bourgain-Li \cite{Bourgain2} and Elgindi-Masmoudi \cite{E-ARMA} proved the non-existence of $C^1$ solutions to the incompressible Euler equations with $C^1$ initial data. Misio{\l}ek and Yoneda \cite{MY} studied the borderline cases and showed that the 2D Euler equations are not locally well-posed in the sense of Hardamard in the $C^1$ space. Bardos and Titi \cite{B} revisited this shear-flow and constructed 3D weak solutions which initially lies in $C^\alpha$ but does not belong to any $C^\beta$ for any $t>0$ and $1>\beta>\alpha^2$ (solutions exhibit instantaneous loss of smoothness in $C^\alpha$ for any $0<\alpha<1$). By this technique, the ill-posedness in the Zygmund space $B_{\infty, \infty}^1$ are obtained in \cite{B}. Misio{\l}ek and Yoneda \cite{MY2} showed that the solution map for the Euler equations is not even continuous in the space of H\"{o}lder continuous functions and thus not locally Hadamard well-posed in $C^{1,\alpha}$ with any $\alpha\in(0,1)$. Our recent work \cite{LY} furthermore gives that the solution map of the Euler equation is not continuous in the Besov spaces from $B^s_{p,\infty}$ to $L_T^\infty B^s_{p,\infty}$ for $s>1+d/p$ with $1\leq p\leq \infty$ and in the H\"{o}lder spaces from $C^{k,\alpha}$ to $L_T^\infty C^{k,\alpha}$ with $k\in \mathbb{N}^+$ and $\alpha\in(0,1)$, which later covers particularly the ill-posedness of $C^{1,\alpha}$ solution in \cite{MY2}.

\begin{itemize}
  \item {\bf The Euler equations with the Coriolis force}
\end{itemize}

Compared with the classical Euler equations, it is surprising that, in the absence of viscosity, the Coriolis force alone is sufficient to stabilize the solutions globally in time in the full 3D setting (see \cite{YGUO}). It is found that the strong rotational effect enhances the temporal decay rate of a certain norm of
the velocity (see \cite{AKL}). As observed in \cite{ch1,ch2,d}, the Euler-Coriolis system \eqref{CE} exhibits a dispersion phenomenon which is due to the presence of the Coriolis force $\Omega e_3\times u$. For large Coriolis parameter $|\Omega|$, Dutrifoy \cite{d} showed the asymptotics of solutions to vortex
patches or Yudovich solutions as the Rossby number goes to zero for some particular initial data. For $u_0 \in H^s(\mathbb{R}^3)$ with $s>5/2$, Koh-Lee-Takada \cite{KLT} proved that there exists a unique local in time solution to \eqref{CE} with $\Omega\in\R$ in the class $\mathcal{C}([0, T]; H^s(\mathbb{R}^3)) \cap \mathcal{C}^1([0, T] ; H^{s-1}(\mathbb{R}^3))$ (see also \cite{T}). Moreover, assuming that $s>7/2$, they showed that their solutions can be extended to long-time intervals $\left[0, T_{\Omega}\right]$ provided that the speed of rotation is large enough. V. Angulo-Castillo and L.C.F Ferreira \cite{CF} further gave extensions to the critical Besov space $B^{5/2}_{2,1}(\R^3)$. Jia-Wan \cite{JW} proved the long time existence of classical solutions to \eqref{CE} for initial data in the
Sobolev space $H^s(s > 5/2)$ with a weaker assumption on the lower bound of $|\Omega|$. Li-Yu-Zhu \cite{BULL} studied the continuous properties for the 3D incompressible rotating Euler equations in Besov spaces $B^{s}_{p,r}(\R^3)$. However, all the above mentioned results for \eqref{CE} are obtained in the spaces based on $L^p$ with $1<p<\infty$. In the cases when $p=\infty$ and $p=1$, due to the fact that Riesz transform does not map continuously from $L^\infty$ to $L^\infty$ and from $L^1$ to $L^1$ \cite{gr}, we will encounter the main difficulty when establishing the uniform bounds of solution in $L^\infty$ and $L^1$-based spaces.
Naturally, we want to ask that

\begin{flushleft}
{\bf Question 1:}\, If $u_0\in B^s_{\infty,r}$, is it true that $u\in L^{\infty}_T(B^s_{\infty,r})$, for even a short time?
\end{flushleft}
\begin{flushleft}
{\bf Question 2:}\, If $u_0\in B^s_{1,r}$, is it true that $u\in L^{\infty}_T(B^s_{1,r})$, for even a short time?
\end{flushleft}

\subsection{Main Results}
The first result of this paper reads as follows:
\begin{theorem}\label{th1}
Let $\Omega\neq 0$. Let $s\in\R$ and $r\in[1,\infty]$. The rotating Euler equations \eqref{CE} is ill-posed in $B^s_{\infty,r}(\R^3)$.
There exists a universal constant $\delta>0$ and a sequence of initial data $u^n_0\in \mathcal{S}(\R^3)$ satisfying
\bbal
\|u^n_0\|_{B^s_{\infty,r}}\approx \|u^n_0\|_{L^\infty}\to 0,
\end{align*}
but which generates a sequence of corresponding smooth solution $u^n(t_n)$ with $t_n\to 0$ of the rotating Euler equations \eqref{CE} satisfying
\bbal
\|u^n(t_n)\|_{B^s_{\infty,r}} \geq \|\Delta_{-1}[u(t_n)]\|_{L^\infty}\geq  \delta>0.
\end{align*}
\end{theorem}
\begin{remark}
Theorem \ref{th1} implies the ill-posedness of \eqref{CE} in any $B^s_{\infty,r}(\R^3)$ in the sense that the solution map to this system is discontinuous at $u_0 = 0$ in the metric of $B^s_{\infty,r}(\R^3)$. To the best of our knowledge, our work is the first one addressing the ill-posedness issue on the 3D rotating Euler equations in $B^s_{\infty,r}$-spaces without the vorticity formulation. The failure of continuity in Theorem \ref{th1} does seem to be related to the mechanism which essentially relies on unboundedness of the Riesz transform in $L^\infty$. Our new observation demonstrates that the instability mechanism is driven by the effect of rotation (the Coriolis force), which contrasts with the recent works \cite{Guoyan,YGUO, tiang} on the stabilizing effect of rotation.
\end{remark}

\begin{remark} Our result implies that the $L^\infty$ ill-posedness of the velocity equation \eqref{CE}. Compared with the seminal work of Elgindi and Masmoudi \cite{E-ARMA}, who established a
framework for characterizing the discontinuity of Riesz operators in $L^{\infty}$ via logarithmic
 singularities and obtained ill-posedness results in the $L^\infty$ framework
for the vorticity equation (Yudovich theory) of the corresponding fluid equations, our work is devoted to establishing a more general framework to study non-linear and non-local transport equations in any $L^\infty$-based spaces (such as $W^{1,\infty}$, $C^{k}$, $C^{k,\alpha}$, etc). We point out that there are several differences beyond purely technical aspects between this paper and \cite{E-ARMA}. In particular, when dealing with the scalar equations in \cite{E-ARMA}, it is not possible to take the transport term as a perturbation
directly since the transport term contains derivatives of solution, thus they worked at
such low regularity and have to ``put the derivative on the velocity" by passing to the Lagrangian
formulation and write the equation along the characteristics of the velocity. In this paper, our
working model is the velocity equations and our method is more direct in a purely $L^\infty$ space without any commutator estimation and the flow map.
\end{remark}
\begin{remark} Compared with the currently-known result that the classical Euler equations is well-posed in $B^s_{\infty,r}$ with $s>1$ and $r\in[1,\infty]$ or $s=1$ and $r=1$, Theorem \ref{th1} is new and surprising since the rotating Euler equations is ill-posed in
$B^s_{\infty,r}$ for any $s\in \R$ and $r\in[1,\infty]$.
This can be seen clearly from the Table below.
\begin{table}[http]
      \centering
      \begin{tabular}{l|c|c}\hline
        \makecell[c]{Spaces}&Euler&rotaing Euler\\\hline
        \makecell[c]{$B_{\infty, r}^{s}(s>1,1\leq r<\infty)$}&WP\cite{B}&IP\\\hline
        \makecell[c]{$B_{\infty, \infty}^s(s>1)$}&IP\cite{LY}&IP\\\hline
        \makecell[c]{$B_{\infty, 1}^{1}$}&WP\cite{GLY,Pak1}&IP\\\hline
        \makecell[c]{$B_{\infty, \infty}^{1}$,\,$B_{\infty, r}^{1}(1<r<\infty)$}&IP\cite{Bardos,SK}&IP\\\hline
        \makecell[c]{$C^{0,\alpha}$,\,$C^{1}$,\,$C^{1,\alpha}(0<\alpha<1)$}&IP\cite{Bardos,Bourgain2,E-ARMA,MY,MY2}&IP\\\hline
        \makecell[c]{$B_{\infty, r}^{s}(s<1)$}&Unknow&IP\\\hline
        \end{tabular}
        \vspace{3mm}
        \caption{Well/Ill-posedness in $L^\infty$-based spaces}
        \end{table}
\end{remark}

In the 3D case the vorticity is vector-valued and given by
$\omega=\text{curl}\,u=\nabla \times u$. The 3D Euler equations with the Coriolis force in vorticity formulation has the form
\begin{align}\label{CE-v}
\begin{cases}
\pa_t \omega+u\cdot \nabla \omega+\omega\cdot \nabla u=\Omega\pa_3(-\Delta)^{-1}\nabla\times \omega, &\quad (t,x)\in \R^+\times\R^3,\\
u=(-\Delta)^{-1}\nabla\times \omega,&\quad (t,x)\in \R^+\times\R^3,\\
\omega(0,x)=\omega_0(x), &\quad x\in \R^3.
\end{cases}
\end{align}
As a direct consequence of Theorem \ref{th1}, we have
\begin{corollary}\label{cor1}
Let $\Omega\neq 0$. Let $s\in\R$ and $r\in[1,\infty]$. The system \eqref{CE-v} is mildly ill-posed in $B^{s}_{\infty,r}(\R^3)$.
\end{corollary}
The second result of this paper implies that the vorticity equation \eqref{CE-v} is ill-posed in any $L^1$-type Besov spaces, which reads as follows:
\begin{theorem}\label{th2}
Let $\Omega\neq 0$. Let $s\in\R$ and $r\in[1,\infty]$. The rotating Euler equations \eqref{CE-v} is ill-posed in $B^s_{1,r}(\R^3)$.
There exists a universal constant $\delta>0$ and a sequence of initial data $w^n_0\in \mathcal{S}(\R^3)$ satisfying
\bbal
\|\omega^n_0\|_{B^s_{1,r}}\approx \|\omega^n_0\|_{L^1}\to 0,
\end{align*}
but which generates a sequence of corresponding smooth solution $w^n(t_n)$ with $t_n\to 0$ of the rotating Euler equations \eqref{CE-v} satisfying
\bbal
\|\omega(t_n)\|_{B^s_{1,r}}\gtrsim \|\Delta_{-1}[\omega(t_n)]\|_{L^1}\geq  \delta.
\end{align*}
\end{theorem}
\begin{corollary}\label{cor2}
The system \eqref{CE-v} is mildly ill-posed in $B^{3}_{1,1}(\R^3)$ and $W^{3,1}(\R^3)$.
\end{corollary}
\begin{remark} It should be mentioned that there exists a few results in $L^1$-type end-point critical spaces for the incompressible Euler equations, e.g., Pak-Park \cite{Pak1} obtained the local existence in the Besov space $B_{1, 1}^{d+1}(\R^d)$ for the velocity equation, and Danchin-Jeong \cite{dj} proved the global well-posedness for the vorticity equation and the propagation of the regularity of the vorticity in the Sobolev space $W^{d,1}(\R^d)$. Corollary \ref{cor2} are in stark contrast to existing well-posedness results as mentioned above in $L^1$-type end-point critical Sobolev spaces.
\end{remark}
\begin{remark}
We would like to mention that, the approach developed in this paper is suitable for more general transport equations of the type:
\begin{align}\label{1w}
\begin{cases}
\pa_tf+u\cdot\na f=\mathcal{R}f, \\
\div\, u=0,\\
f(0,x)=f_0(x),
\end{cases}
\end{align}
where $u$ is is a divergence-free Lipschitz function and $\mathcal{R}$ is a (double) Riesz operator.
\subsection{Strategies to the proof}

In order to introduce the main idea, we consider the model \eqref{1w} and summarize the roadmap for establishing the mild ill-posedness (for example in $L^\infty$) in the following three steps:

{\bf Step 1: Establishing local-in-time well-posedness.}  We show that any smooth initial data $f_0$, belonging to the critical space $\dot{B}^0_{\infty,1}\cap B^1_{\infty,1}$, generates a unique local-in-time solution $f$ of \eqref{1w} over the time of existence $T_n\approx \frac{1}{\|f_0\|_{\dot{B}^0_{\infty,1}\cap B^1_{\infty,1}}}$.

{\bf Step 2: Constructing initial data.}  We analyze the Riesz operator $\mathcal{R}$ and show that construct a family of initial data such that  $\|f^n_0\|_{B_{\infty,r}^s}$ is small enough and $\|\mathcal{R}f^n_0\|_{L^\infty}\approx \|f^n_0\|_{\dot{B}^0_{\infty,1}}$ is arbitrarily large.

{\bf Step 3: Estimating the lower bound of solution in $L^\infty$-type spaces.}  We show that $t_n\mathcal{R}f^n_0$ possess the lower bound in $L^\infty$-type spaces which leads to the solution map exhibits discontinuity in timescale $t_n= \frac{\ep}{\|f^n_0\|_{\dot{B}^0_{\infty,1}}}$ for small $\ep>0$.

Lastly, we should emphasize that, the approach developed in a purely $L^\infty (L^1)$-based framework relies on the unboundedness of the Riesz transform in $L^\infty (L^1)$ and the linear mechanism, and thus does not extend beyond mild ill-posedness although the nonlinear term is treated as a perturbation and is harmless. We refer to see the strong ill-posedness results in \cite{BHI,E-APDE}.
\end{remark}

\subsection{Organization of our paper}

In Section \ref{sec2}, we list some notations and known results which will be used in the sequel.
applications will follow. In Section \ref{sec3}, we will prove Theorem \ref{th1} based on the fact that Riesz transform does not map continuously
from $L^\infty$ to $L^\infty$. In Section \ref{sec4}, we will prove Theorem \ref{th2} based on the fact that Riesz transform does not map continuously from $L^1$ to $L^1$.

\section{Preliminaries}\label{sec2}
We will use the following notations throughout this paper.
\begin{itemize}
\item We write functions depending on time and space as $u(t,x)$ and partial derivatives in time and space are respectively
denoted by $\pa_tu$ and $\pa_{i}u=\pa_{x_i}u$, where $i = 1,\ldots,d$. The metric $\nabla u$ denotes the gradient of $u$ with respect to the $x$ variable, whose $(i,j)$-th component is given by $(\nabla u)_{ij}=\pa_iu_j$ with $1\leq i,j\leq d$.
  \item For $X$ a Banach space and $I\subset\R$, we denote by $\mathcal{C}(I;X)$ the set of continuous functions on $I$ with values in $X$. Sometimes we will denote $L^p(0,T;X)$ by $L_T^p(X)$.
      \item We will also define the Lipschitz space $C^{0,1}$ using the norm $\|f\|_{C^{0,1}}=\|f\|_{L^\infty}+\|\nabla f\|_{L^\infty}$.
  \item The symbol $\mathrm{A}\lesssim (\gtrsim)\mathrm{B}$ means that there is a uniform positive ``harmless" constant $\mathrm{C}$ independent of $\mathrm{A}$ and $\mathrm{B}$ such that $\mathrm{A}\leq(\geq) \mathrm{C}\mathrm{B}$, and we sometimes use the notation $\mathrm{A}\approx \mathrm{B}$ which  means that $\mathrm{A}\lesssim \mathrm{B}$ and $\mathrm{B}\lesssim \mathrm{A}$. Also, $\mathrm{A}=O(\mathrm{B})$ means that $\f|\frac{\mathrm{A}}{\mathrm{B}}\g|\lesssim 1$.
  \item Let $\mathcal{B}(x,r) = \{y\in \R^d: |y-x| < r\}$ denote the open ball centered at $x$ with radius $r$ and $\mathcal{C}(r_1,r_2)=\{x\in \R^d: 0<r_1\leq|x|\leq r_2,\;0 <r_1 <r_2\}$ be an annulus centered at $0$.
  \item We use $\mathcal{S}(\R^d)$ to denote the Schwartz functions spaces on $\R^d$, which is defined by
$$
\mathcal{S}(\mathbb{R}^d)=\left\{f \in C^{\infty}(\mathbb{R}^d):\sup_{|\alpha|\leq k,\, x\in\R^d}\f(1+|x|\g)^k\f|\partial^{\alpha} f(x)\g| <\infty,\, \forall\, k \in \mathbb{N}\right\}.
$$
 \item Let us recall that for all $u\in \mathcal{S}$, the Fourier transform $\mathcal{F}u$, also denoted by $\widehat{u}$, is defined by
$$
(\mathcal{F}u)(\xi)=\widehat{u}(\xi)=(2\pi)^{-\frac{d}{2}}\int_{\R^d}e^{-\mathrm{i}x\cd \xi}u(x)\dd x \quad\text{for any}\; \xi\in\R^d.
$$
  \item The inverse Fourier transform $\mathcal{F}^{-1}$, also denoted by $\check{u}$, is defined by
$$
\check{u}(x)=(\mathcal{F}^{-1}u)(x)=(2\pi)^{-\frac{d}{2}}\int_{\R^d}e^{\mathrm{i}x\cdot\xi}u(\xi)\dd\xi.
$$
\item The pseudo-differential operator is defined by $\sigma(D):u\to\mathcal{F}^{-1}(\sigma \mathcal{F}u)$.
 \item We denote the Leray projection
\bbal
&\mathbf{P}: L^{p}(\mathbb{R}^{d}) \rightarrow L_{\sigma}^{p}(\mathbb{R}^{d}) \equiv \overline{\left\{f \in \mathcal{C}^\infty_{0}(\mathbb{R}^{d}) ; {\rm{div}} f=0\right\}}^{\|\cdot\|_{L^{p}(\mathbb{R}^d)}},\quad p\in(1,\infty),\\
&\mathbf{Q}=\mathrm{Id}-\mathbf{P}.
\end{align*}
In $\mathbb{R}^{d}$, $\mathbf{P}$ can be defined by $\mathbf{P}= \mathrm{Id}+(-\Delta)^{-1}\nabla {\rm{div}}$, or equivalently, $\mathbf{P}=(\mathbf{P}_{i j})_{1 \leqslant i, j \leqslant d}$, where $\mathbf{P}_{i j} \equiv \delta_{i j}+\mathcal{R}_{i} \mathcal{R}_{j}$ with $\delta_{i j}$ being the Kronecker delta ($\delta_{i j}=0$ for $i\neq j$ and $\delta_{i i}=0$) and $\mathcal{R}_{j}$ being the Riesz transform with symbol $-\mathrm{i}\xi_j/|\xi|$.
\end{itemize}
Next, we will recall some facts about the Littlewood-Paley (L-P) decomposition, the homogeneous Besov spaces and their some useful properties (see \cite{B} for more details).
\begin{proposition}[L-P decomposition, \cite{B}]\label{po} Let $\mathcal{B}:=\{\xi\in\mathbb{R}^d:|\xi|\leq 4/3\}$ and $\mathcal{C}:=\{\xi\in\mathbb{R}^d: 3/4\leq|\xi|\leq 8/3\}$.
Choose a radial, non-negative, smooth function $\vartheta:\R^d\mapsto [0,1]$ such that
\bbal
\vartheta(\xi)=
\bca
1, \quad \mathrm{if} \ |\xi|\leq \frac{3}{4},\\
0, \quad \mathrm{if} \ |\xi|\geq \frac{4}{3}.
\eca
\end{align*}
Setting
$\phi(\xi):=\vartheta(\xi/2)-\vartheta(\xi)$,
then we deduce that $\vartheta$ and $\phi$ satisfy the following properties
\begin{itemize}
  \item ${\rm{supp}} \;\vartheta\subset \mathcal{B}$ and ${\rm{supp}} \;\phi\subset \mathcal{C}$;
  \item $\vartheta(\xi)\equiv1$ for $|\xi|\leq3/4$ and $\phi(\xi)\equiv 1$ for $4/3\leq |\xi|\leq 3/2$;
  \item $\vartheta(\xi)+\sum_{j\geq0}\phi(2^{-j}\xi)=1$ for any $\xi\in \R^d$;
  \item $\sum_{j\in \mathbb{Z}}\phi(2^{-j}\xi)=1$ for any $\xi\in \R^d\setminus\{0\}$.
\end{itemize}
\end{proposition}
The nonhomogeneous and homogeneous dyadic blocks are defined as follows
\begin{align*}
\forall\, u\in \mathcal{S'}(\R^d),\quad {\Delta_ju=}
\begin{cases}
0, & \text{if}\ j\leq-2,\nonumber\\
\vartheta(D)u, & \text{if}\ j=-1,\nonumber\\
\phi(2^{-j}D)u, & \text{if}\ j\geq0,\nonumber
\end{cases}
\end{align*}
and
\begin{align*}
\forall\, u\in \mathcal{S}'_h(\R^d),\quad
\dot{\Delta}_ju=\phi(2^{-j}D)u,\; \; \text{if}\;j\in \mathbb{Z},
\end{align*}
where $\mathcal{S}'_h$ is given by
\begin{eqnarray*}
\mathcal{S}'_h:=\Big\{u \in \mathcal{S'}(\mathbb{R}^d):\; \lim_{j\rightarrow-\infty}\|\vartheta(2^{-j}D)u\|_{L^{\infty}}=0 \Big\}.
\end{eqnarray*}
We recall the definition of the Besov spaces and norms.
\begin{definition}[Nonhomogeneous Besov spaces, see \cite{B}]
Let $s\in\mathbb{R}$ and $(p,r)\in[1, \infty]^2$. We define the nonhomogeneous Besov spaces
$$
B^{s}_{p,r}:=\f\{f\in \mathcal{S}':\;\|f\|_{B^{s}_{p,r}}:=\left\|2^{js}\|\Delta_jf\|_{L_x^p}\right\|_{\ell^r(j\geq-1)}<\infty\g\}.
$$
\end{definition}
\begin{definition}[Homogeneous Besov spaces, see \cite{B}]
Let $s\in\mathbb{R}$ and $(p,r)\in[1, \infty]^2$. We define the homogeneous Besov spaces
$$
\dot{B}^{s}_{p,r}:=\f\{f\in \mathcal{S}'_h:\;\|f\|_{\dot{B}^{s}_{p,r}}:=\left\|2^{js}\|\dot{\Delta}_jf\|_{L_x^p}\right\|_{\ell^r(j\in \mathbb{Z})}<\infty\g\}.
$$
\end{definition}
\begin{remark} For any $s>0$ and $(p,r)\in[1, \infty]^2$, it holds that $B^{s}_{p,r}(\R^3)=\dot{B}^{s}_{p,r}\cap L^p(\R^3)$ and
$$\|f\|_{B^{s}_{p,r}}\approx \|f\|_{L^{p}}+\|f\|_{\dot{B}^{s}_{p,r}}.$$
In particular, for $s=0$, it holds that for $p\in[1, \infty]$
$$\B^0_{p,1}\hookrightarrow B^0_{p,1}\hookrightarrow L^p\hookrightarrow B^0_{p,\infty}.$$
\end{remark}
Next we recall the following product law which will be used often in the sequel.
\begin{lemma}[\cite{B}]\label{lem21}
Assume that $s>0$ and $(p,r)\in(1,\infty)\times[1,\infty]$. Let $X=B^{s}_{p,r}(\mathbb{R}^d)$ or $X=\dot{B}^{s}_{p,r}(\mathbb{R}^d)$. Then
 there exists a constant $C$, depending only on $d,p,r,s$ such that
$$\|fg\|_{X}\leq C\f(\|f\|_{L^\infty}\|g\|_{B^s_{p,r}}+\|g\|_{L^\infty}\|f\|_{X}\g),\quad \forall f,g\in L^\infty\cap X.$$
\end{lemma}
The following Bernstein's inequalities will be used in the sequel.
\begin{lemma}[\cite{B}] \label{lem22} Let $\mathcal{B}$ be a ball and $\mathcal{C}$ be an annulus. There exists a constant $C>0$ such that for all $k\in \mathbb{N}\cup \{0\}$, any $\lambda\in \R^+$ and any function $f\in L^p$ with $1\leq p \leq q \leq \infty$, we have
\begin{align*}
&{\rm{supp}}\ \widehat{f}\subset \lambda \mathcal{B}\;\Rightarrow\; \|D^kf\|_{L^q}\leq C^{k+1}\lambda^{k+(\frac{1}{p}-\frac{1}{q})}\|f\|_{L^p},  \\
&{\rm{supp}}\ \widehat{f}\subset \lambda \mathcal{C}\;\Rightarrow\; C^{-k-1}\lambda^k\|f\|_{L^p} \leq \|D^kf\|_{L^p} \leq C^{k+1}\lambda^k\|f\|_{L^p}.
\end{align*}
\end{lemma}
The following corollaries are straightforward.
\begin{corollary} Let $1\leq q \leq p \leq \infty$ and $1\leq r_1 \leq r_2 \leq \infty$, then for any $s\in\R$, it holds that $\dot{B}_{q, r_1}^{s}(\R^d)\hookrightarrow \dot{B}_{p, r_2}^{s-d\left(\frac{1}{q}-\frac{1}{p}\right)}(\R^d)$ since
$$
\|\dot{\Delta}_j u\|_{L^p(\R^d)} \leq C2^{d\left(\frac{1}{q}-\frac{1}{p}\right) j}\|\dot{\Delta}_j u\|_{L^q(\R^d)}.
$$
\end{corollary}
\begin{corollary}  We have the following equivalence
$$
\|\nabla^k f\|_{\dot{B}_{p, q}^s(\R^d)}\approx\|f\|_{\dot{B}_{p, q}^{s+k}(\R^d)}.
$$
\end{corollary}
\begin{corollary}
For $s>d/ p$ with $1 \leq p, q \leq \infty$, or $s=d / p$ with $1 \leq p \leq \infty$ and $q=1$, we have the estimate
$$
\|f\|_{L^{\infty}(\R^d)} \leq C\|f\|_{B_{p, q}^s(\R^d)},
$$
Thus, for $s> d/ p+1$ with $1 \leq p, q \leq \infty$ or $s=d / p+1$ with $1 \leq p \leq \infty$ and $q=1$, we have the estimates
$$
\|\nabla f\|_{L^{\infty}(\R^d)} \leq C\|\nabla f\|_{B_{p, q}^{s-1}(\R^d)} \leq C\|f\|_{B_{p, q}^s(\R^d)}.
$$
\end{corollary}

Finally, we recall the regularity estimate for the transport equations.
\begin{lemma}[\cite{B}]\label{zgj}
Assume that $s\geq0$ and $1\leq p,r\leq \infty$.  Let $X=B^{s}_{p,r}(\mathbb{R}^d)$ or $X=\dot{B}^{s}_{p,r}(\mathbb{R}^d)$ and $f$ be any regular solution of the transport equations on some time interval $[0,t]$
\begin{align*}
\begin{cases}
\pa_t f+u\cdot \nabla f=g\in L_{\rm{loc }}^1\left(\mathbb{R}^{+}; X\right), \\
\div\,u=0,\\
f(0,x)=f_0(x)\in X.
\end{cases}
\end{align*}
Then it holds that for some constant $C=C(s,p,q)>0$
\begin{align*}
\|f\|_{L^\infty_t(X)}\leq \|f_0\|_{X}+C\int^t_0\|\nabla u(\tau)\|_{L^\infty}\|f(\tau)\|_{X}\dd \tau+\|g\|_{L^1_t(X)}.
\end{align*}
\end{lemma}

\section{Ill-posedness in $L^\infty$ type Besov spaces}\label{sec3}

In this section we aim to proving Theorem \ref{th1} based on the fact that Riesz transform does not map continuously
from $L^\infty$ to $L^\infty$.
\subsection{Local-in-time well-posedness}
First we establish the following well-posedness result for the 3D rotating Euler equations in $L^\infty$-type spaces. Since the concrete values of the coefficients $\Omega\neq 0$ have no impact on the result, from now on we set $\Omega=1$ in \eqref{CE} for simplicity.
\begin{lemma}\label{lem3-3}
Let $u_0 \in \B^0_{\infty,1}(\mathbb{R}^3)\cap B^1_{\infty,1}(\mathbb{R}^3)$ with $\div\, u_0=0$. There exists some $T=T(\left\|u_0\right\|_{\B^0_{\infty,1}\cap B^1_{\infty,1}})>0$ such that the 3D rotating Euler equations \eqref{CE} has a unique solution $u$ satisfying $u \in C([0, T]; \B^0_{\infty,1}\cap B^1_{\infty,1})$.
Moreover, it holds
\bbal
\|u\|_{L^\infty_T\f(\B^0_{\infty,1}\cap B^1_{\infty,1}\g)}\leq C\|u_0\|_{\B^0_{\infty,1}\cap B^1_{\infty,1}}.
\end{align*}
\end{lemma}

\begin{proof}
Now let's return to the original system
\bal\label{cm}
\partial_t u+u\cdot \nabla u=-(e_3\times u)-\nabla P,,\quad u(t=0)=u_0.
\end{align}
we have
\bal\label{cm1}
\|u(t)\|_{B^1_{\infty,1}}&\leq \|u_0\|_{B^1_{\infty,1}}+C\int^t_0\|u(\tau)\|_{B^1_{\infty,1}}\f(1+\|u(\tau)\|_{B^1_{\infty,1}}\g)\dd \tau\no\\
&\quad+C\int^t_0\|\nabla P(\tau)\|_{B^1_{\infty,1}}\dd \tau.
\end{align}
Due to $$\nabla P=\nabla(-\Delta)^{-1}\div(u\cdot \nabla u+e_3\times u)=\nabla(-\Delta)^{-1}(\pa_iu\pa_ju)+\mathbf{Q}(e_3\times u),$$
then
$$\|\nabla P\|_{B^1_{\infty,1}}\les \|u\|^2_{B^1_{\infty,1}}+\|u\|_{B^1_{\infty,1}}+\|u\|_{\B^0_{\infty,1}\cap B^1_{\infty,1}}.$$
Inserting the above into \eqref{cm1} yields
\bal\label{cm2}
\|u(t)\|_{B^1_{\infty,1}}&\leq \|u_0\|_{B^1_{\infty,1}}+C\int^t_0\|u(\tau)\|_{B^1_{\infty,1}}\f(1+\|u(\tau)\|_{B^1_{\infty,1}}\g)\dd \tau\no\\
&\quad+C\int^t_0\|u(\tau)\|_{\B^0_{\infty,1}\cap B^1_{\infty,1}}\dd \tau.
\end{align}
Applying the Leray projection $\mathbf{P}$ to \eqref{cm}, we have
\bal\label{cm3}
u(t)&=u_0-\int^t_0\mathbf{P} \f(u\cd \na u+e_3\times u\g)(\tau)\dd \tau,
\end{align}
which gives that
\bal\label{cm4}
\|u(t)\|_{\B^0_{\infty,1}}\leq \|u_0\|_{\B^0_{\infty,1}}+C\int^t_0\|u(\tau)\|_{\B^0_{\infty,1}}\f(1+\|u(\tau)\|_{B^1_{\infty,1}}\g)\dd \tau.
\end{align}
Combining \eqref{cm2} and \eqref{cm4}, we obtain
\bbal
\|u(t)\|_{\B^0_{\infty,1}\cap B^1_{\infty,1}}\leq \|u_0\|_{\B^0_{\infty,1}\cap B^1_{\infty,1}}+C\int^t_0\|u(\tau)\|_{\B^0_{\infty,1}\cap B^1_{\infty,1}}\f(1+\|u(\tau)\|_{\B^0_{\infty,1}\cap B^1_{\infty,1}}\g)\dd \tau.
\end{align*}
Solving the above inequality, we deduce that
\begin{align*}
\|u(t)\|_{\B^0_{\infty,1}\cap B^1_{\infty,1}}&\leq \frac{\|u_0\|_{\B^0_{\infty,1}\cap B^1_{\infty,1}}}{1-Ct\|u_0\|_{\B^0_{\infty,1}\cap B^1_{\infty,1}}}\quad\text{whenever}\quad t<\frac{1}{C\|u_0\|_{\B^0_{\infty,1}\cap B^1_{\infty,1}}}.
\end{align*}
By the standard procedure, we complete the proof of Lemma \ref{lem3-3}.
\end{proof}
\subsection{Key Lemmas} First, we construct a class of low frequency functions which have the unbounded property for the $Riesz$ operator in $L^\infty$.

\begin{lemma}\label{lem3-1}
For $d\geq 1$. Assume that $\chi\in \mathcal{S}(\R^d)$ satisfying that $\chi(0)=0$ and $\mathrm{supp}\ \hat{\chi}(\xi) \subset\mathcal{C}(4/3,3/2)=\{\xi\in\mathbb{R}^d: 4/3\leq|\xi|\leq 3/2\}$. Define
\bal\label{fn}
f_n(x)=\frac{1}{\ln n}\sum^{-1}_{k=-n}\chi_k(x)\quad\text{with}\quad \chi_k(\cdot)=\chi(2^{k}\cdot).
\end{align}
Then we have for $p\in[1,\infty]$
\bbal
&\|f_n\|_{\B^{d/p}_{p,1}(\R^d)}\approx \frac{n}{\ln n}\quad \text{but} \quad
\|f_n\|_{L^{\infty}(\R^d)}\les \frac{1}{\ln n}.
\end{align*}
\end{lemma}
\begin{proof}
Due to the fact $\phi(2^{-j}\xi)\equiv 1$ in the annulus $\mathcal{C}(\frac{4}{3}2^{j},\frac{3}{2}2^{j})$ and $\mathcal{F}({\Delta}_j\chi_k)=\phi(2^{-j}\cdot)\widehat{\chi_k}$ for all $j\geq-1$,
then we have
$
\mathcal{F}({\Delta}_j\chi_k)=0$ for  $j\neq k,
$
and thus
$
{\Delta}_j\chi_k=
\chi_k$ if $j=k$. In particular, ${\Delta}_{-1}\chi_k={\Delta}_{0}\chi_k=0.$ Since $\chi$ is a Schwartz function, we have for $M\geq 100d$
\bal\label{s}
|\chi(x)|+|\chi'(x)|\leq C(1+|x|)^{-M}, \quad \forall x \in \mathbb{R}^d.
\end{align}
Using the definition of Besov space and the fact that $\chi$ is a Schwartz function, yields
\bbal
\|f_n\|_{\B^{d/p}_{p,1}(\R^d)}&\approx\frac{1}{\ln n}\sum_{j=-n}^{-1}2^{\fr{d}{p}j}\f\|\chi(2^{j}x)\g\|_{L^p(\R^d)}\approx \frac{n}{\ln n}.
\end{align*}
 We next prove the second result. Obvious, $\chi(0)=0$. We should emphasize that the fact is crucial. We divide the proof into two cases.

{\bf Case when $|x|\leq 1$.} Using \eqref{s}, we have
\bal\label{f1}
\ln n\cdot|f_n(x)|&\leq \sum^{-1}_{k=-n}|\chi(2^{k}x)-\chi(0)|\leq \sum^{-1}_{k=-n}C|2^{k}x|
\leq   C\sum^{-1}_{k=-n}2^{k} \leq C.
\end{align}

{\bf Case when $|x|\in[1,\infty)$.} Picking large $k_0$ such that $1\leq 2^{-k_0}|x|\leq 2$, then using \eqref{s}, we have
\bal\label{f2}
\ln n\cdot|f_n(x)|&\leq \sum^{-1}_{k=-k_0}|\chi(2^{k}x)|+ \sum^{-k_0}_{k=-n}|\chi(2^{k}x)-\chi(0)|\no\\
&\leq C\sum^{-1}_{k=-k_0}\frac{C}{1+2^{k}|x|}+ \sum^{-k_0}_{k=-n}|2^{k}x|\no\\
&\leq C\sum^{-1}_{k=-k_0}\frac{C}{1+2^{k+k_0}}+  \sum^{-k_0}_{k=-n}2^{k+k_0}\leq C.
\end{align}
Combining \eqref{f1} and \eqref{f2}, we complete the proof of Lemma \ref{lem3-1}.
\end{proof}

\begin{lemma}\label{lem3-2}
 For $d\geq 1$. Let $1\leq i\leq d$ and $f_n(x)$ be given by \eqref{fn} with $\chi=\pa_i\check{\varphi}$ if $d\geq 1$ or $\chi=x_i\check{\varphi}$ if $d\geq 2$, where $\check{\varphi}$ is the inverse Fourier transform of $\varphi$ and $\varphi:\R^d\mapsto [0,1]$ is a radial, non-negative, smooth function satisfying $\mathrm{supp}\ {\varphi}(\xi) \subset\mathcal{C}(4/3,3/2)$. Then we have
\bbal
&\|f_n\|_{L^\infty(\R^d)}\les \frac{1}{\ln n}\quad \text{but} \quad \|\mathcal{R}_if_n\|_{L^\infty(\R^d)}\approx \frac{n}{\ln n}.
\end{align*}
In particular,
\bbal
&\frac{n}{\ln n}\les\|\mathcal{R}_if_n\|_{L^\infty(\R^d)}\les\|f_n\|_{\B^{0}_{\infty,1}(\R^d)}\les \frac{n}{\ln n},
\end{align*}
where $\mathcal{R}_i$ is the $i$-th component of the classical Riesz transform and is given on the Fourier transform side by multiplication by the function $-\mathrm{i}\frac{\xi_i}{|\xi|}$.
\end{lemma}

\begin{proof}
The first result is obvious. It is easy to show that
\bbal
\|\mathcal{R}_if_n\|_{L^\infty}\leq C\|\mathcal{R}_if_n\|_{\B^{0}_{\infty,1}}\leq C\|f_n\|_{\B^{0}_{\infty,1}}\leq C \frac{n}{\ln n}.
\end{align*}
Hence, we just need to prove
\bbal
\|\mathcal{R}_if_n\|_{L^\infty}\geq c \frac{n}{\ln n}.
\end{align*}
We divide the proof into two cases.

{\bf Case 1: $\chi=\pa_i\check{\varphi}$.}\,
Due to $\widehat{\mathcal{R}_i}=-\mathrm{i}\frac{\xi_i}{|\xi|}$, one has
\bbal
\mathcal{R}_i f_n(x)=\frac{1}{\ln n}\sum^{-1}_{k=-n}\gamma(2^kx)\quad\text{where}\quad \hat{\gamma}(\xi)=\frac{\xi^2_i}{|\xi|}{\varphi}(\xi).
\end{align*}
Notice that $\hat{\gamma}(\xi)=\hat{\gamma}(|\xi|)$ and recall that $\rm{supp} \;{\varphi}\subset \mathcal{C}(4/3,3/2)$, one has
\bbal
\gamma(0)&=\int_{\R^d}\frac{\xi^2_i}{|\xi|}{\varphi}(\xi)\dd \xi=\frac1d\int_{\R^d}|\xi|{\varphi}(\xi)\dd \xi=\tilde{c}>0.
\end{align*}
Thus
\bbal
\|\mathcal{R}_if_n\|_{L^\infty}&\geq |\mathcal{R}_if_n(x=0)|=\gamma(0)\frac{1}{\ln n}\sum^{-1}_{k=-n}1\geq \tilde{c}\frac{n}{\ln n}.
\end{align*}

{\bf Case 2: $\chi=x_i\check{\varphi}$.}\,
Due to $\widehat{\mathcal{R}_i}=-\mathrm{i}\frac{\xi_i}{|\xi|}$, one has
\bbal
\mathcal{R}_i f_n(x)=\frac{1}{\ln n}\sum^{-1}_{k=-n}\gamma(2^kx)\quad\text{where}\quad \hat{\gamma}(\xi)=\frac{\xi_i}{|\xi|}\pa_{\xi_i}\varphi(\xi).
\end{align*}
Notice that $\hat{\gamma}(\xi)=\hat{\gamma}(|\xi|)$, then integrating by parts, one has
\bbal
-\gamma(0)&=-\int_{\R^d}\frac{\xi_i}{|\xi|}\pa_{\xi_i}\varphi(\xi)\dd \xi
=\int_{\R^d}\pa_{\xi_i}\f(\frac{\xi_i}{|\xi|}\g)\varphi(\xi)\dd \xi\\
&=\int_{\R^d}\frac{|\xi|^2-\xi^2_i}{|\xi|^3}\varphi(\xi)\dd \xi
=\frac{d-1}{d}\int_{\R^d}\frac{1}{|\xi|}\varphi(\xi)\dd \xi\geq \tilde{c}>0.
\end{align*}
Thus
\bbal
\|\mathcal{R}_if_n\|_{L^\infty}&\geq |\mathcal{R}_if_n(x=0)|=-\gamma(0)\frac{1}{\ln n}\sum^{-1}_{k=-n}1\geq \tilde{c}\frac{n}{\ln n}.
\end{align*}
This completes the proof of Lemma \ref{lem3-2}.
\end{proof}

\subsection{Construction of initial data}

We choose
$$f_{n}(x)=\Gamma_n\sum^{-1}_{k=-n}2^{-k}(\pa_3\check{\varphi})(2^{k}x)\quad \text{with}\quad \Gamma_n=\frac{1}{\ln n},$$
where $\check{\varphi}$ is the inverse Fourier transform of $\varphi$ and  $\varphi:\R^3\mapsto [0,1]$  is a radial, non-negative, smooth function satisfying $\mathrm{supp}\ {\varphi}(\xi) \subset\{\xi\in\mathbb{R}^3: 4/3\leq|\xi|\leq 3/2\}$.
Obviously, $f_n$ is a real scalar function.
We construct the initial data $u^n_0$ whose components are given by
\begin{align}\label{u0-de}
u^n_0=(\pa_2f_n,-\pa_1f_n,0).
\end{align}

Obviously, it is easy to verify that $\div\, u^n_0=0$ and initial data $u^n_0$ is real-valued Schwarz functions.

{\bf Some Key Observations}
\begin{enumerate}
  \item Let $a=\pa_{3}\check{\varphi}$ and $a_k(x)=(\pa_{3}\check{\varphi})(2^kx)$, a trivial computation gives that
\bbal
\mathrm{supp} \ \widehat{a_k}(\xi)&\subset  \left\{\xi\in\R^3: \ \frac{4}{3}2^{k}\leq |\xi|\leq \frac{3}{2}2^{k}\right\},\quad k\in[-n,-1]\\
&\subset  \left\{\xi\in\R^3: \ |\xi|\leq \frac{3}{4}\right\},
\end{align*}
which in turn gives
\bal\label{o1}
&\mathrm{supp} \ \widehat{u^{n}_{0}}\subset \left\{\xi\in\R^3: \ |\xi|\leq \frac{3}{4}\right\}.
\end{align}
  \item It is not difficult to check that
  \bbal
{\Delta}_{j}f_{n}=
\bca
f_{n}, \quad &\mathrm{if} \ j=-1,\\
0, \quad &\mathrm{if} \ j\geq0,
\eca
\end{align*}
and
\bbal
{\dot{\Delta}}_{j}f_{n}=
\bca
\Gamma_n2^{-j}(\pa_3\check{\varphi})(2^{j}x), \quad &\mathrm{if} \ j\in[-n,-1],\\
0, \quad &\mathrm{if} \ j\notin[-n,-1].
\eca
\end{align*}
 \item It is easy to find the crucial fact for any $s\in\R$ and $r\in[1,\infty]$
$$\|f_n\|_{B^{s}_{\infty,r}}=2^{-s}\|f_n\|_{L^\infty}.$$
\end{enumerate}

\subsection{Estimation of initial data}
\begin{proposition}\label{pro3-1}
For any $s\in\R$ and $r\in[1,\infty]$. Let $u^n_0$ be defined by \eqref{u0-de}. Then there exists $C,c>0$ independent of $n$ such that
\bal
&\|u^n_0\|_{B^s_{\infty,r}}\approx \|u^n_0\|_{L^\infty}\leq C\Gamma_n, \label{u1}\\
&\|u^n_0\|_{\B^0_{\infty,1}} \approx n\Gamma_n,\label{u2}\\
&\|\Delta_{-1}[\mathbf{Q}(e_3\times u^n_0)]\|_{L^\infty(\R^3)}\geq  cn\Gamma_n.
\end{align}
\end{proposition}

\begin{proof}
Due to the fact $\phi(2^{-j}\xi)\equiv 1$ in the annulus $\mathcal{C}_j=\left\{\xi\in\R^3: \frac{4}{3}2^{j}\leq |\xi|\leq \frac{3}{2}2^{j}\right\}$ and $\mathcal{F}({\Delta}_ja_k)=\phi(2^{-j}\cdot)\widehat{a_k}$ for all $j\geq-1$,
then we have
$
\mathcal{F}({\Delta}_ja_k)=0$ for  $j\neq k,
$
and thus
$
{\Delta}_ja_k=
a_k$ if $j=k$. In particular, ${\Delta}_{-1}a_k={\Delta}_{0}a_k=0.$ Since $\chi$ is a Schwartz function, we have for $M\geq 100d$
\bal\label{s1}
|a(x)|+|a'(x)|\leq C(1+|x|)^{-M}, \quad \forall x \in \mathbb{R}^3.
\end{align}
For $i=1,2$, we have
\bal\label{u3}
\pa_if_n(x)=\Gamma_n\sum^{-1}_{k=-n}\chi(2^{-k}x)\quad\text{with}\quad \chi=\pa_i\pa_3\check{\varphi}.
\end{align}
By Lemma \ref{lem3-1}, we have
\bbal
\|u^n_{0}\|_{L^\infty}&\leq \|\pa_1f_n\|_{L^\infty}+\|\pa_2f_n\|_{L^\infty}
\les\Gamma_n.
\end{align*}
Using the definition of Besov space yields that
\bbal
\|u^n_0\|_{B^{s}_{\infty,r}(\R^3)}&=2^{-s}\|u^n_0\|_{L^\infty(\R^3)}\les\Gamma_n.
\end{align*}
This gives \eqref{u1}.
Using the definition of Besov space and the fact that $\chi$ is a Schwartz function, we obtain from \eqref{u3} that
\bbal
\|\pa_if_n\|_{\B^{0}_{\infty,1}(\R^3)}
=\Gamma_n\sum_{j=1}^{n}\f\|\chi(2^{-j}x)\g\|_{L^\infty(\R^3)}\approx n\Gamma_n,
\end{align*}
which gives that
\bbal
\|u^n_0\|_{\B^{0}_{\infty,1}(\R^3)}&\approx n\Gamma_n.
\end{align*}
Notice that $e_3 \times u^n_0=(\pa_1f_n, \pa_2f_n, 0)$, then $\div (e_3\times u^n_0)=(\pa^2_1+\pa^2_2) f_n$ and thus we focus on the third component
\bbal
F(x)=-\pa_3(-\Delta)^{-1}\div (e_3\times u^n_0)
&=\mathcal{F}^{-1}\f(\frac{\mathrm{i}\xi_3(\xi^2_1+\xi^2_2)}{|\xi|^2}\widehat{f_n}\g).
\end{align*}
Notice that
$$\widehat{f_n}(\xi)=\Gamma_n\sum^{-1}_{k=-n}2^{4k}(-\mathrm{i}\xi_3\varphi)(2^{k}\cdot),$$
then
\bbal
F(x=0)&=\int_{\R^3}\widehat{F}(\xi)\dd \xi=\Gamma_n\sum^{-1}_{k=-n}\int_{\R^3}2^{5k}\frac{\xi^2_3(\xi^2_1+\xi^2_2)}{|\xi|^2}\varphi(2^k\xi)\dd \xi
\\&=\Gamma_n\sum^{-1}_{k=-n}\int_{\R^3}\frac{\xi^2_3(\xi^2_1+\xi^2_2)}{|\xi|^2}\varphi(\xi)\dd \xi=:c_0n\Gamma_n.
\end{align*}
Combining the above, we deduce that
\bbal
\|\Delta_{-1}[\mathbf{Q}(e_3\times u^n_0)]\|_{L^{\infty}(\R^3)}&=\|\mathbf{Q}(e_3\times u^n_0)\|_{L^{\infty}(\R^3)}\\
&\geq\|\pa_3(-\Delta)^{-1}\div (e_3\times u^n_0)\|_{L^\infty}\\
&\geq F(x=0)=c_0n\Gamma_n.
\end{align*}
This completes the proof of Proposition \ref{pro3-1}.
\end{proof}

\subsection{Discontinuity of the solution map}

By Lemma \ref{lem3-3}, we can deduce that there exists $T_n \approx\frac{1}{n\Gamma_n}$ such that $u_n(t,x)\in C([0,T_n];\B^0_{\infty,1}\cap B^1_{\infty,1})$ be the solution of the 3D rotating Euler equations \eqref{CE} with initial data $u^n_0$. Moreover,  we also have
\bbal
\|u_n\|_{L^\infty_{T_n}\f(\B^0_{\infty,1}\cap B^1_{\infty,1}\g)}\leq  C \frac{n}{\ln n},
\end{align*}
which in turn gives that (see \eqref{cm2})
 \bal\label{unn}
\|u_n\|_{L^\infty_{T_n}(B^1_{\infty,1})}\leq  \|u^n_0\|_{B^1_{\infty,1}}+CT_n\cdot \frac{n}{\ln n}\leq C.
\end{align}
Thus there exists $T_n \approx\frac{1}{n\Gamma_n}$ such that $u_n(t,x)\in C([0,T_n]; \bigcap_{s\geq0}B^s_{\infty,r})$ with $r\in[1,\infty]$ be the solution of the 3D rotating Euler equations \eqref{CE} with initial data $u^n_0$.

{\it Proof of Theorem \ref{th1}.}\; From now on, we take $t\in(0,T_n]$. Applying the projection $\mathbf{P}$ to both sides of the first equation of \eqref{CE}, we obtain the following equations:
\begin{align*}
\begin{cases}
\frac{\partial u_n}{\partial t} + \mathbf{P} (e_3 \times u_n) + \mathbf{P}(u_n \cdot \nabla u_n) = 0, \\
\text{div } u_n = 0, \\
u_n(0, x) = u^n_0(x).
\end{cases}
\end{align*}
Let us define the projection \( \mathbb{P}_{\pm} \) by
$$
\mathbb{P}_{\pm} u_n:= \frac{1}{2} (\mathbf{P} u_n\mp \mathcal{R} \times u_n)\quad\text{with}\quad\widehat{\mathcal{R}}:=-\mathrm{i}\frac{\xi}{|\xi|},
$$
it is easy to check the following relation holds:
$$
\mathbf{P}(e_{3} \times u_n)=\mathcal{R}_{3}(\mathbb{P}_{+} - \mathbb{P}_{-})u_n,
$$
Denote the operator $\mathcal{L}=-\mathcal{R}_{3}(\mathbb{P}_{+} - \mathbb{P}_{-})$, by the Duhamel principle we obtain from \eqref{CE}
\bal\label{vhh}
u_n(t)&=e^{\mathcal{L}t}u^n_0-\int^t_0e^{\mathcal{L}\tau}\mathbf{P}(u_n \cdot \nabla u_n)(\tau)\dd \tau.
\end{align}
From \eqref{vhh}, one has
\bbal
\|u_n(t)-e^{\mathcal{L}t}u^n_0\|_{\dot{B}^0_{\infty,1}}&\leq C\int^t_0\|(u_n\otimes u_n)(\tau)\|_{\dot{B}^1_{\infty,1}}\dd \tau
\\&\leq C\int^t_0\|u_n(\tau)\|_{L^\infty}\|u_n(\tau)\|_{\dot{B}^1_{\infty,1}}\dd \tau
\\&\leq C\int^t_0\|u_n(\tau)\|^2_{B^1_{\infty,1}}\dd \tau\quad \text{by}\; \eqref{unn}
\\&\leq Ct,
\end{align*}
which implies that
\bal\label{uhhh}
\|u_n(t)\|_{B^s_{\infty,r}}\geq c\|\Delta_{-1}u_n(t)\|_{L^\infty}\geq c\f\|\Delta_{-1}[e^{\mathcal{L}t}u_{0}^n]\g\|_{L^\infty}-Ct.
\end{align}
By Taylor's formula
$$e^{\mathcal{L}t}=\mathrm{Id}+\mathcal{L}t+t^2\sum^{\infty}_{k=1}\frac{t^{k-2}\mathcal{L}^k}{k!},$$
one has
\bbal
\|(e^{\mathcal{L}t}-\mathrm{Id}-\mathcal{L}t)u_{0}^n\|_{\dot{B}^0_{\infty,1}}&\leq Ct^2\|u_{0}^n\|_{\dot{B}^0_{\infty,1}}.
\end{align*}
From \eqref{uhhh}, we have
\bbal
\|u_n(t)\|_{B^s_{\infty,r}}&\geq ct\f\|\Delta_{-1}[\mathcal{L}u_{0}^n]\g\|_{L^\infty}-\|u_{0}^n\|_{L^\infty}-Ct^2\|u_{0}^n\|_{\dot{B}^0_{\infty,1}}-Ct.
\end{align*}
Recall that $\mathcal{L}u_{0}^n=\mathbf{P} (e_3 \times u_{0}^n)=-\mathbf{Q} (e_3\times u_{0}^n)+e_3 \times u_{0}^n$, we deduce that
\bal\label{hh1}
\|u_n(t)\|_{B^s_{\infty,r}}&\geq ct\f\|\Delta_{-1}[\mathbf{Q} (e_3\times u_{0}^n)]\g\|_{L^\infty}-C(1+t)\|u_{0}^n\|_{L^\infty}-Ct^2\|u_{0}^n\|_{\dot{B}^0_{\infty,1}}-Ct.
\end{align}
By the construction of $u^n_0$ (see \eqref{u0-de}), from Proposition \ref{pro3-1}, one has
\bbal
&\|u^n_0\|_{\B^0_{\infty,1}}\leq Cn\Gamma_n, \\
&\|u^n_0\|_{B^1_{\infty,1}}\leq C, \\
&\|u^n_0\|_{L^\infty}\leq C\Gamma_n,\\
&\|\Delta_{-1}[\mathbf{Q} (e_3\times u^n_0)]\|_{L^\infty}\geq cn\Gamma_n,
\end{align*}
where $C$ and $c$ are universal constants.

Then we see from \eqref{hh1} that for $n\gg 1$
\bbal
\|u_n(t)\|_{B^s_{\infty,r}}&\geq c t n\Gamma_n-C(1+t) \Gamma_n
-Ct^2n\Gamma_n-Ct.
\end{align*}
Let $t_n=\frac{\ep}{n\Gamma_n}$, from the above, picking $\ep$ small enough and $n$ large enough, we deduce that
$$
\|u(t_n)\|_{B^s_{\infty,r}}
\geq c\varepsilon-C\Gamma_n-C\ep t_n-Ct_n\geq \frac{c}{2}\varepsilon>0.
$$
This completes the proof of Theorem \ref{th1}.

\section{Ill-posedness in $L^1$ type Besov spaces}\label{sec4}

In this section we aim to proving Theorem \ref{th2} based on the fact that Riesz transform does not map continuously
from $L^1$ to $L^1$.
\subsection{Local-in-time well-posedness}
First we establish the following well-posedness result for the 3D rotating Euler equations with $\Omega=1$ in $L^1$-type spaces.
\begin{lemma}\label{j-th1}
Let $\omega_0 \in \B^0_{1,1}(\mathbb{R}^3)\cap B^3_{1,1}(\mathbb{R}^3)$. There exists some $T=T(\left\|\omega_0\right\|_{\B^0_{1,1}\cap B^3_{1,1}})>0$ such that the 3D rotating Euler equations \eqref{CE-v} has a unique solution $u$ satisfying $\omega \in C([0, T]; \B^0_{1,1}\cap B^3_{1,1})$.
Moreover, it holds
\bbal
\|\omega\|_{L^\infty_T\f(\B^0_{1,1}\cap B^3_{1,1}\g)}\leq C\|\omega_0\|_{\B^0_{1,1}\cap B^3_{1,1}}.
\end{align*}
\end{lemma}

\begin{proof}
Now let's return to the original vorticity formulation \eqref{CE-v}
\bal\label{cm-w}
\pa_t \omega+u\cdot \nabla \omega=\pa_3(-\Delta)^{-1}\nabla\times \omega-\omega\cdot \nabla u,\quad \omega(t=0)=\omega_0.
\end{align}
Using Lemma \ref{zgj}, we have
\bal\label{cmw1}
\|\omega(t)\|_{B^3_{1,1}}&\leq \|\omega_0\|_{B^3_{1,1}}+C\int^t_0\|\omega(\tau)\|_{\B^0_{1,1}\cap B^3_{1,1}}\f(1+\|\omega(\tau)\|_{B^3_{1,1}}\g)\dd \tau.
\end{align}
Also, we have
\bal\label{cmw2}
\|\omega(t)\|_{\B^0_{1,1}}\leq \|\omega_0\|_{\B^0_{\infty,1}}+C\int^t_0\|\omega(\tau)\|_{\B^0_{1,1}}\f(1+\|\omega(\tau)\|_{\B^0_{1,1}}\g)\dd \tau.
\end{align}
Combining \eqref{cmw1} and \eqref{cmw2}, we obtain
\bbal
\|\omega(t)\|_{\B^0_{1,1}\cap B^3_{1,1}}\leq \|\omega_0\|_{\B^0_{1,1}\cap B^3_{1,1}}+C\int^t_0\|\omega(\tau)\|_{\B^0_{1,1}\cap B^3_{1,1}}\f(1+\|\omega(\tau)\|_{\B^0_{1,1}\cap B^3_{1,1}}\g)\dd \tau.
\end{align*}
Solving the above inequality, we deduce that
\begin{align*}
\|\omega(t)\|_{\B^0_{1,1}\cap B^3_{1,1}}&\leq \frac{\|\omega_0\|_{\B^0_{1,1}\cap B^3_{1,1}}}{1-Ct\|\omega_0\|_{\B^0_{1,1}\cap B^3_{1,1}}}\quad\text{whenever}\quad t<\frac{1}{C\|\omega_0\|_{\B^0_{1,1}\cap B^3_{1,1}}}.
\end{align*}
By the standard procedure, we complete the proof of Lemma \ref{j-th1}.
\end{proof}
\subsection{Key Lemmas} First, we construct a class of low frequency functions which have the unbounded property for $Riesz$ operator in $L^1$.

\begin{lemma}[Large decomposition of Riesz operator]\label{lem4-1}
Letting $f\in \mathcal{S}(\R^d)$ with $\int_{\R^d}f\dd x=1$. We define as \cite[Definition 5.1.13]{gr}
$$\mathcal{R}_jf(x) = \mathrm{p.v.}\int_{\mathbb R^d}K_j(x-y)f(y)\dd y,$$
where the kernel is
$$K_j(z)=c_d\frac{z_j}{|z|^{d+1}}, \qquad c_d=\frac{\Gamma\big(\tfrac{d+1}{2}\big)}{\pi^{(d+1)/2}}.$$
Then it holds that
\bbal
\mathcal{R}_jf(x)=
\bca
c_d\frac{x_j}{|x|^{d+1}}+O\f(|x|^{-d-1}\g), \quad &\mathrm{if} \ |x|\geq2,\\
O(1), \quad &\mathrm{if} \ |x|\leq2.
\eca
\end{align*}
\end{lemma}

\begin{proof}
First of all we decompose $\mathcal{R}_jf(x)$ as
\begin{align*}
\mathcal{R}_jf(x)&= \int_{\Omega_1}K_j(x-y)f(y)\dd y+\int_{\Omega_2}K_j(x-y)f(y)\dd y+\mathrm{p.v.}\int_{\Omega_3}K_j(x-y)f(y)\dd y,\\
&=I_1+I_2+I_3,
\end{align*}
where we divide $y\in \R^d$ into tree parts, i.e.,
\begin{align*}
&\Omega_1:=\f\{y\in \R^d:\,|y|\leq  \frac12|x|\g\},\\
&\Omega_2:=\f\{y\in \R^d:\,|y|> 2|x|\g\},\\
&\Omega_3:=\f\{y\in \R^d:\,\frac12|x|\leq |y|\leq 2|x|\g\}.
\end{align*}
Next, we will show that the first term $I_1$ contributes the main part.

{\bf For first term $I_1$.}  The kernel has no singularity in this case. Obvious, we have $|x-y|\geq \frac12 |x|$ when $y\in \Omega_1$.
By Taylor's formula, we have
$$K_j(x-y)=K_j(x)-\sum_{l=1}^d y_l\partial_l K_j(x)+R_2(x,y),$$
where the remainder term satisfies
$$R_2(x,y)=\frac{1}{2}\sum_{l,m=1}^d y_ly_m(\partial_m\partial_l K_j)\left(x-\theta y\right),\;\theta\in(0,1).$$
Notice that
\begin{align*}
&K_j(x)=c_d\frac{x_j}{|x|^{d+1}}, \\
&\partial_lK_j(x) = c_d|x|^{-d-3}\f(\delta_{jl}|x|^{2}-(d+1)x_jx_l\g),\\
&\partial_{ml} K_j(x) =-c_d(d+1) |x|^{-d-3} \f( \delta_{jl} x_m + \delta_{mj} x_l + x_j \delta_{ml} - (d+3) \frac{x_j x_l x_m}{|x|^2} \g),
\end{align*}
thus
\begin{align*}
&|\pa_lK_j(x)|\leq C|x|^{-d-1},\quad |\partial_{ml} K_j(x)|\leq C|x|^{-d-2},
\end{align*}
which implies that for $y\in \Omega_1$
$$|R_2(x,y)|\leq C|y|^2|x|^{-d-2}\leq C|y\|x|^{-d-1}.$$
From the above, one has
\begin{align*}
\f|I_1-K_j(x)\g|&\leq \int_{\Omega^c_1}|f(y)|\dd y+C|x|^{-d-1}\int_{\Omega_1}|y\|f(y)|\dd y.\end{align*}
Since $f$ is a Schwartz function, we have for $N\geq 2d+2$
\bbal
|f(x)|+|\nabla f(x)|\leq C(1+|x|)^{-N}\leq C|x|^{-N}, \quad \forall x \in \mathbb{R}^d.
\end{align*}
thus
$$I_1 = c_d\frac{x_j}{|x|^{d+1}} + O(|x|^{-d-1}).$$

{\bf For second term $I_2$.}  The kernel also has no singularity in this case. Obvious, we have $|x-y|\geq \frac12|y|\geq |x|$ when $y\in \Omega_2$, which gives that $|K_j(x-y)|\leq |x-y|^{-d} \leq  2^d|y|^{-d}$. Thus we have
$$|I_2|\leq C\int_{\Omega_2}|y|^{-d-N}dy=O(|x|^{-N}).$$

{\bf For third term $I_3$.} We should mention that, the kernel has a singularity in $\Omega_3$. We further divider $\Omega_3$ into two parts:
$\Omega_3=\Omega'_3\cup \Omega''_3$, where
 $$\Omega'_3=\Omega_3\cap\f\{y:\,|x-y|\geq\fr12|x|\g\},\quad \Omega''_3=\f\{y:\,|x-y|\leq\fr12|x|\g\}.$$
 For $y\in \Omega'_3$, notice that $f(y) = O(|y|^{-N})$, we have
$$\int_{\Omega'_3}|K_j(x-y)f(y)|\dd y\leq C|x|^{-N-d}\int_{\Omega_3}\dd y\leq C|x|^{-N}.$$
For $y\in \Omega''_3$, we rewrite $$K_j(x-y)f(y)=K_j(x-y)f(x)+K_j(x-y)[f(y)-f(x)].$$
Obvious,
$$\mathrm{p.v.}\int_{\Omega''_3}K_j(x-y)f(y)\dd y=\mathrm{p.v.}\int_{\Omega''_3}K_j(x-y)[f(y)-f(x)]\dd y.$$
In fact, we have used
\begin{align*}
\mathrm{p.v.}\int_{\Omega''_3}K_j(x-y)\dd y&
=\lim_{\ep\to0^+}\int_{\mathcal{C}(\ep,\fr12|x|)}\frac{z_i}{|z|^{d+1}}\dd z=0.
\end{align*}
Using the Newton-Leibniz formula, we have
$$f(y)-f(x)=\int_0^1(y-x)\cdot(\nabla f)(x+\theta(y-x))\dd\theta,$$
which implies that for $y\in \Omega''_3$
\begin{align*}
|f(y)-f(x)|\leq C|y-x|\int_0^1(1+|x+\theta(y-x)|)^{-N}\dd\theta\leq  C|y-x\|x|^{-N},
\end{align*}
and thus
\begin{align*}
\int_{\Omega''_3}|K_j(x-y)[f(y)-f(x)]|\dd y\leq C|x|^{-N}\int_{\mathcal{B}(0, 3|x|)}|z|^{-d+1}\dd z\leq C|x|^{1-N},
\end{align*}
which gives that
\begin{align*}
I_3\leq C|x|^{1-N}.
\end{align*}
In summary, combing the above three cases and choosing $N\geq 2d+2$, we have $$\mathcal{R}_jf(x)=c\frac{x_j}{|x|^{d+1}}+O(|x|^{-d-1}),\quad |x|\geq2.$$
For $|x|\leq2$, we decompose
\begin{align*}
\mathcal{R}_jf(x)&=\int_{|x-y|\geq1}K_j(x-y)f(y)\dd y+\mathrm{p.v.}\int_{|x-y|\leq1}K_j(x-y)f(y)\dd y.
\end{align*}
Following the above similar procedure, we obtain
\begin{align*}
\f|\int_{|x-y|\geq1}K_j(x-y)f(y)\dd y\g|&\leq \int_{|x-y|\geq1}|f(y)|\dd y\leq \int_{\R^d}|f(y)|\dd y\leq C,
\end{align*}
and
\begin{align*}
\f|\mathrm{p.v.}\int_{|x-y|\leq1}K_j(x-y)f(y)\dd y\g|&=\f|\mathrm{p.v.}\int_{|x-y|\leq1}K_j(x-y)[f(y)-f(x)]\dd y\g|
\\&\leq C\|\na f\|_{L^\infty}\int_{|x-y|\leq1}\frac{1}{|x-y|^{d-1}}\dd y \leq C.
\end{align*}
This implies $|\mathcal{R}_jf(x)|\leq C$ for $|x|\leq2$. Thus we complete the proof of Lemma \ref{lem4-1}.
\end{proof}

Similarly, for the double Riesz operator, we have
\begin{lemma}[Large decomposition of double Riesz operator]\label{lem4-2}
Assume that $1\leq i\neq j\leq d$. Letting $f\in \mathcal{S}(\R^d)$ with $\int_{\R^d}f\dd x=1$. Define
$$\mathcal{R}_{ij}f(x) = \mathrm{p.v.}\int_{\mathbb R^d}K_{ij}(x-y)f(y)\dd y,$$
where the kernel is
$$K_{ij}(z)=c'_d\frac{z_iz_j}{|z|^{d+2}}, \quad c'_d=\frac{\Gamma\big(\tfrac{d+2}{2}\big)}{\pi^{d/2}}.$$
Then it holds that
\bbal
\mathcal{R}_{ij}f(x)=
\bca
c'_d\frac{x_ix_j}{|x|^{d+1}}+O\f(|x|^{-d-1}\g), \quad &\mathrm{if} \ |x|\geq2,\\
O(1), \quad &\mathrm{if} \ |x|\leq2.
\eca
\end{align*}
\end{lemma}

\begin{corollary}\label{co44}
Assume that $1\leq i\neq j\leq d$. Letting $f\in \mathcal{S}(\R^d)$ with $\int_{\R^d}f\dd x=1$. Define $f_n(x)=2^{-dn}f(2^{-n}x)$ with $1\ll n\in \mathbb{N}^+$. Then, we have
\bbal
\mathcal{R}_{j}f_n(x)=
\begin{cases}
2^{-dn}\cdot O(1), \quad &\mathrm{if} \ 2\leq |x|\leq 2^{n+1} ,\\
c_d\frac{x_j}{|x|^{d+1}}+2^{-dn}\cdot O\f((|2^{-n}x|)^{-d-1}\g),\quad &\mathrm{if} \ |x|\geq 2^{n+1},
\end{cases}
\end{align*}
and
\bbal
\mathcal{R}_{ij}f_n(x)=
\begin{cases}
2^{-dn}\cdot O(1),\quad &\mathrm{if} \ 2\leq |x|\leq 2^{n+1},\\
c'_d\frac{x_ix_j}{|x|^{d+2}}+2^{-dn}\cdot O\f((|2^{-n}x|)^{-d-1}\g),\quad &\mathrm{if} \ |x|\geq 2^{n+1}.
\end{cases}
\end{align*}
\end{corollary}
\begin{proof}
Obvious, we have $$\mathcal{R}_{j}f_n(\cdot)=2^{-dn}(\mathcal{R}_{j}f)(2^{-n}\cdot)\quad\text{and}\quad \mathcal{R}_{ij}f_n(\cdot)=2^{-dn}(\mathcal{R}_{ij}f)(2^{-n}\cdot).$$
Then using Lemmas \ref{lem4-1}-\ref{lem4-2}, we complete the proof.
\end{proof}

\subsection{Construction of initial data}
We choose
$$g_n(x)=\Gamma_n\f[2^{-6}\check{\vartheta}(2^{-2}x)-2^{-3n}\check{\vartheta}(2^{-n}x)\g]\quad\text{with}\quad \Gamma_n=\frac{1}{\ln n},\;n\gg1,$$
where $\vartheta$ is the non-negative smooth function given in Proposition \ref{po}. Obviously, $g_n$ is a real scalar function.
We construct the initial data $\omega^n_0$ whose components are given by
\begin{align}\label{gn}
\omega^n_0=(g_n,0,0).
\end{align}

{\bf Some Key Observations}
\begin{enumerate}
  \item A trivial computation gives that $\widehat{g_n}(\xi)=\Gamma_n\f[\vartheta(4\xi)-\vartheta(2^{n}\xi)\g]$ and thus
\bbal
\mathrm{supp} \ \widehat{g_n}(\xi)&\subset  \left\{\xi\in\R^3: \ \frac342^{-n}\leq |\xi|\leq \frac13\right\}\subset \left\{\xi\in\R^3: \ |\xi|\leq \frac{1}{3}\right\},
\end{align*}
which implies that for $n\gg1$
\bbal
{\Delta}_{j}g_{n}=
\bca
g_{n}, \quad &\mathrm{if} \ j=-1,\\
0, \quad &\mathrm{if} \ j\geq 0.
\eca
\end{align*}
  \item It is not difficult to check that for $n\gg1$
\bbal
{\dot{\Delta}}_{j}g_{n}=
\bca
\Gamma_n\f(2^{3j}\check{\phi}(2^{j}x)-{\dot{\Delta}}_{j}[2^{-3n}\check{\vartheta}(2^{-n}x)]\g), \quad &\mathrm{if} \ j\in \{-n-1,-n\},\\
\Gamma_n2^{3j}\check{\phi}(2^{j}x), \quad &\mathrm{if} \ j\in[-n+1,-4],\\
\Gamma_n{\dot{\Delta}}_{j}[2^{-6}\check{\vartheta}(2^{-2}x)], \quad &\mathrm{if} \ j\in \{-3,-2\},\\
0, \quad &\mathrm{if} \ j\geq -1 \quad \mathrm{or} \quad  j\leq -n-2.
\eca
\end{align*}
\item It holds that for any $s\in\R$ and $r\in[1,\infty]$
$$\|g_n\|_{B^{s}_{1,r}}=2^{-s}\|g_n\|_{L^1}.$$
\end{enumerate}

\subsection{Estimation of initial data}
\begin{proposition}\label{pro4-1}
For any $s\in \R$ and $r\in[1,\infty]$. Let $\omega^n_0$ be defined by \eqref{gn}. Then there exists $C,c>0$ independent of $n$ such that
\bal
&\|\omega^n_0\|_{B^s_{1,r}}\approx \|\omega^n_0\|_{L^1}\leq C\Gamma_n, \label{g1}\\
&\|\omega^n_0\|_{\B^0_{1,1}} \approx n\Gamma_n,\label{g2}\\
&\f\|\Delta_{-1}[\pa_3(-\Delta)^{-1}\nabla\times \omega^n_0]\g\|_{L^1}\geq  cn\Gamma_n.\label{g3}
\end{align}
\end{proposition}

\begin{proof} Obviously, one has $\|\check{\vartheta}\|_{L^1(\R^3)}=\vartheta(0)=1$ and thus $\|g_n\|_{L^1}\leq 2\Gamma_n$.
Base on the above key observations, we have for any $s\in \R$
\bbal
&\|\omega_{0}^n\|_{B^s_{1,r}}=\sum_{j\geq-1}2^{sj}\|{\Delta}_j\omega_{0}^n\|_{L^1}=2^{-s}\|\omega_{0}^n\|_{L^1}=2^{-s}\|g_n\|_{L^1}\leq 2^{1-s}\Gamma_n,
\end{align*}
and
\bbal
&\|\omega_{0}^n\|_{\dot{B}^0_{1,1}}=\sum_{j=-n-1}^{0}\|\dot{\Delta}_jg_n\|_{L^1}\leq Cn\Gamma_n, \\
&\|\omega_{0}^n\|_{\dot{B}^0_{1,1}}\geq \sum_{j=-n+1}^{-4}\|\dot{\Delta}_jg_n\|_{L^1}=\Gamma_n\sum_{j=-n+1}^{-4}\|2^{3j}\check{\phi}(2^{j}x)\|_{L^1}\geq cn\Gamma_n.
\end{align*}
Notice that $\nabla\times \omega^n_0=(0, \pa_3g_n, -\pa_2 g_n)$, we focus on the third component and obtain from Corollary \ref{co44} that
\bbal
\|\Delta_{-1}[\pa_3(-\Delta)^{-1}\nabla\times \omega^n_0]\|_{L^1(\R^3)}&=\|\pa_3(-\Delta)^{-1}\nabla\times \omega^n_0\|_{L^1(\R^3)}
\\&\geq \|\mathcal{R}_{23} g_n\|_{L^1(\R^3)}
\\&\geq \|\mathcal{R}_{23} g_n\|_{L^1(\mathcal{C}(8,2^n))}
\\&\geq c\Gamma_n\int_{\mathcal{C}(8,2^n)}\frac{|x_2x_3|}{|x|^{d+2}}\dd x-C\Gamma_n
\\&\geq cn\Gamma_n-C\Gamma_n\geq cn\Gamma_n.
\end{align*}
This completes the proof of Proposition \ref{pro4-1}.
\end{proof}
\begin{remark}
We would like to emphasize that, our construction of the initial data is not suitable for the velocity due to its incompressibility.
\end{remark}
\subsection{Discontinuity of the solution map}
By Lemma \ref{j-th1}, we can deduce that there exists $T_n \approx\frac{1}{n\Gamma_n}$ such that $\omega_n(t,x)\in C([0,T_n];\B^0_{1,1}\cap B^3_{1,1})$ be the solution of the 3D rotating Euler equations \eqref{CE-v} with initial data $\omega^n_0$. Moreover, we also have
\bbal
\|\nabla u_n\|_{L^\infty_{T_n}(L^\infty)}\leq  C \|\omega\|_{L^\infty_{T_n}(\B^0_{1,1})}\leq  C \frac{n}{\ln n}.
\end{align*}
Using Lemma \ref{zgj} again, we have
 \bbal
\|\omega_n\|_{L^\infty_{T_n}(B^3_{1,1})}\leq  \|u^n_0\|_{B^3_{1,1}}+CT_n\cdot \frac{n}{\ln n}\leq C.
\end{align*}
Thus there exists $T_n \approx\frac{1}{n\Gamma_n}$ such that $\omega(t,x)\in C([0,T_n]; \bigcap_{s\geq0}B^s_{1,r})$ with $s\geq0,r\in[1,\infty]$ be the solution of the 3D rotating Euler equations \eqref{CE-v} with initial data $\omega^n_0$.

{\it Proof of Theorem \ref{th2}.}\; From now on, we take $t\in(0,T_n]$ and denote $\mathcal{\tilde{R}}=\pa_3(-\Delta)^{-1}\nabla\times $. Taking advantage of the Duhamel formula, then we have from \eqref{cm-w}
\bal\label{hh}
\omega_n(t)&=e^{\mathcal{\tilde{R}}t}\omega^n_0-\int^t_0e^{\mathcal{\tilde{R}}\tau}\mathrm{div}\f(u_n\otimes \omega_n+\omega_n\otimes u_n\g)(\tau)\dd \tau.
\end{align}
From which, one has
\bbal
\|\omega_n(t)-e^{\mathcal{\tilde{R}}t}\omega^n_0\|_{\dot{B}^0_{1,1}}&\leq C\int^t_0\|(u_n\otimes \omega_n)(\tau)\|_{\dot{B}^1_{1,1}}\dd \tau
\\&\leq C\int^t_0\|u_n(\tau)\|_{L^\infty}\|\omega_n(\tau)\|_{\dot{B}^1_{1,1}}
+\|\omega_n(\tau)\|_{L^2}\|u_n(\tau)\|_{\dot{B}^1_{2,1}}\dd \tau
\\&\leq C\int^t_0\|\omega_n(\tau)\|_{\dot{B}^1_{1,1}}\|\omega_n(\tau)\|_{\dot{B}^2_{1,1}}
+\|\omega_n(\tau)\|_{L^2}\|\omega_n(\tau)\|_{\dot{B}^{\frac32}_{1,1}}\dd \tau
\\&\leq C\int^t_0\|\omega_n(\tau)\|^2_{B^3_{1,1}}\dd \tau
\\&\leq Ct,
\end{align*}
which implies that
\bal\label{hhh}
\|\omega_n(t)\|_{B^s_{1,r}}\geq c\|\Delta_{-1}\omega_n\|_{L^1}\geq c\f\|\Delta_{-1}[e^{\mathcal{\tilde{R}}t}\omega_{0}^n]\g\|_{L^1}-Ct.
\end{align}
By Taylor's formula
$$e^{\mathcal{\tilde{R}}t}=\mathrm{Id}+\mathcal{\tilde{R}}t+t^2\sum^{\infty}_{k=1}\frac{t^{k-2}\mathcal{\tilde{R}}^k}{k!},$$
one has
\bbal
\|(e^{\mathcal{\tilde{R}}t}-\mathrm{Id}-\mathcal{\tilde{R}}t)\omega_{0}^n\|_{\dot{B}^0_{1,1}}&\leq Ct^2\|\omega_{0}^n\|_{\dot{B}^0_{1,1}}.
\end{align*}
From \eqref{hhh}, we have
\bbal
\|\omega_n(t)\|_{B^s_{1,r}}&\geq ct\f\|\Delta_{-1}[\mathcal{\tilde{R}}\omega_{0}^n]\g\|_{L^1}-C\|\omega_{0}^n\|_{L^1}-Ct^2\|\omega_{0}^n\|_{\dot{B}^0_{1,1}}-Ct.
\end{align*}
Using Proposition \ref{pro4-1}, we deduce that
\bbal
\|\omega_n(t)\|_{B^s_{1,r}}&\geq cnt\Gamma_n-C\Gamma_n-Ct^2\Gamma-Ct.
\end{align*}
Let $t_n=\frac{\ep}{n\Gamma_n}$ with $\ep\in(0,1)$, from the above, picking $\ep$ small enough and $n$ large enough, we deduce that
\bbal
\|\omega_{n}(t_n)\|_{L^1}&\geq \frac{c}{2}\ep>0.
\end{align*}
This completes the proof of Theorem \ref{th2}.

\section*{Declarations}
\noindent\textbf{Data Availability}\\
No data was used for the research described in the article.

\vspace*{1em}
\noindent\textbf{Conflict of interest}\\
The authors declare that they have no conflict of interest.

\section*{Acknowledgements}

J. Li was supported by Ganpo Talents Project of Jiangxi Province (gpyc20240069) and Natural Science Foundation of Jiangxi Province (20252BAC210004 and 20252BAC240186).

\end{document}